\documentclass[preprint,12pt]{elsarticle}

\usepackage{amsfonts,color,amsmath,amssymb,fancyhdr}
\usepackage{amsthm}
\usepackage{tikz}
\usepackage{graphicx}
\usepackage{subfigure}
\usepackage{geometry}
\newtheorem{thm}{Theorem}

\newtheorem{lemma}[thm]{Lemma}

\numberwithin{equation}{section}
\numberwithin{thm}{section}

\newtheorem{remark}[thm]{Remark}\theoremstyle{definition}

\begin{document}

\begin{frontmatter}




\newcommand{\symfont}{\fam \mathfam}

\title{An infinite family of $m$-ovoids of $Q(4,q)$}
\author[add1]{Tao Feng}
\ead{tfeng@zju.edu.cn}
\author[add1]{Ran Tao\corref{cor1}}
\ead{rant@zju.edu.cn}
\cortext[cor1]{Corresponding author}
\address[add1]{School of Mathematical Sciences, Zhejiang University, 38 Zheda Road, Hangzhou 310027, Zhejiang P.R China}

\begin{abstract}
In this paper, we construct an infinite family of $\frac{q-1}{2}$-ovoids of the generalized quadrangle $Q(4,q)$, for $q\equiv 1 (\text{mod}\ 4)$ and $q>5$. Together with \cite{feng2016family} and \cite{bamberg2009tight}, this establishes the existence of $\frac{q-1}{2}$-ovoids in $Q(4,q)$ for each odd prime power $q$.
\end{abstract}




\end{frontmatter}


\section{Introduction}
A \textit{generalized quadrangle} (GQ) of order $(s,t)$ is an incidence structure of points and lines with the properties that every two points are incident with at most one line, every point is incident with $t+1$ lines, every line is incident with $s+1$ points, and for any point $P$ and line $\ell$ that are not incident there is a unique point on $\ell$ collinear with $P$. The point-line dual of a GQ of order $(s,t)$ is a GQ of order $(t,s)$. In the case $s=t$, we say that the GQ has order $s$. We are only concerned with thick generalized quadrangles, i.e. those GQs of order $(s,t)$ with $s>1$ and $t>1$. The classical GQs are the point-line incidence structures arising from the finite classical polar spaces of rank 2. For more background on GQ, please refer to the monograph \cite{MR2508121}.

In this paper, we are concerned with $m$-ovoids of the classical GQ $Q(4,q)$. The points and lines of $Q(4,q)$ are respectively the totally singular points and totally singular lines contained in a parabolic quadric of $\text{PG}(4,q)$. An ovoid of $Q(4,q)$ is a point set of $Q(4,q)$ that intersects each totally singular line (i.e. generator) in exactly one point. Ovoids of $Q(4,q)$ are of great importance in finite geometry. For instance, its point-line dual gives rise to spreads of $W(3,q)$, which further give rise to translation planes of order $q^2$ by the Bruck-Bose/Andr\'{e} construction. The list of known ovoids in $Q(4,q)$ is very short. Please refer to \cite{penttila2000ovoids} for a summary of known ovoids of $Q(4,q)$. The concept of $m$-ovoid is first introduced by Thas \cite{thas1989interesting}, as a generalization of ovoids. The notion of $m$-ovoids is now known as a special case of intriguing sets, the latter first introduced in \cite{bamberg2009tight} by Bamberg et al. for generalized quadrangles and later generalized to finite polar spaces in \cite{bamberg2007tight}. In short, an $m$-ovoid of $Q(4,q)$ is a point set of $Q(4,q)$ that intersects each totally singular line in exactly $m$ points. It is known that such point sets can give rise to strongly regular graphs and projective two-weight codes, cf. \cite{bamberg2009tight,calderbank1986geometry}.

We now make a summary of known constructions of $m$-ovoids in $Q(4,q)$. For $q$ even, the GQ $W(3,q)$ is isomorphic to $Q(4,q)$. Cossidente et al. \cite{cossidente2008m} have shown that $W(3,q)$ has $m$-ovoids for all integers $m$, $1\leq m\leq q$. Therefore, $Q(4,q)$ also has $m$-ovoids for all integers $m$, $1\leq m\leq q$. For $q$ odd, Cossidente and Penttila  \cite{cossidente2005hemisystems} constructed a hemisystem of $H(3,q^2)$ for all odd $q$ and subsequently there have been more constructions of hemisystems in $H(3,q^2)$ cf. \cite{bamberg2018new, Korchm2017hemisystem, bamberg2010hemisystems}. By duality, this gives rise to a $\frac{q+1}{2}$-ovoid of $Q^-(5,q)$, whose intersection with a non-tangent hyperplane yields a $\frac{q+1}{2}$-ovoid in $Q(4,q)$. In \cite{feng2016family}, the first author and collaborators constructed the first infinite family of $m$-ovoid of $Q(4,q)$ with $m=\frac{q-1}{2}$ for $q\equiv 3 (\text{mod}\ 4)$, which generalized some sporadic examples listed in \cite{bamberg2009tight}. The following table taken from \cite{bamberg2009tight} lists the known $m$-ovoids for small $q$.

In the end of \cite{feng2016family}, it is commented that ``As for future research, it would be interesting to generalize the examples of $\frac{q-1}{2}$-ovoids of $Q(4,q)$ (with $q=5$ or 9) in Example 5 of \cite{bamberg2009tight} into an infinite family. We could not see any general pattern for the prescribed automorphism groups in those examples.'' In this paper, we are able to construct an infinite family of $\frac{q-1}{2}$-ovoids of $Q(4,q)$ for $q\equiv 1 (\text{mod}\ 4)$ and $q>5$. The construction and proof technique in this paper is very similar to that in \cite{feng2016family}. The main difficulty is to choose the right prescribed automorphism group out of the rich subgroup structures of the group $\text{PGO}(5,q)$. Our family does not seem to generalize the sporadic examples listed in \cite{bamberg2009tight} by examining their automorphism groups.

\begin{table}
\begin{center}
\begin{tabular}{|l| l |l|}
\hline
\textbf{$q$} & \textbf{Known $m$} & \textbf{Unknown $m$}\\
\hline
3 & 1,2,3 & - \\
5 & 1,2,3,4,5 & - \\
7 & 1,3,4,5,7 & 2,6 \\
9 & 1,3,4,5,6,7,9 & 2,8 \\
11 & 1,5,6,7,11 & 2,3,4,8,9,10 \\
\hline
\end{tabular}
\end{center}
\caption{m-ovoids of $Q(4,q)$}
\label{tab1}
\end{table}

\section{The model and the prescribed group}
\subsection{A model for $Q(4,q)$}
Let $q\equiv 1 (\text{mod}\ 4)$ be a prime power, and set $V=\mathbb{F}_{q}\times\mathbb{F}_{q}\times\mathbb{F}_{q^2}\times\mathbb{F}_{q}$. We regard $V$ as a 5-dimensional vector space over $\mathbb{F}_{q}$, and write its element in the form $(x,y,\alpha,z)$, where $x,y,z\in\mathbb{F}_{q}$ and $\alpha\in\mathbb{F}_{q^2}$. We define the quadratic form $Q$ on $V$ as follows $$Q((x,y,\alpha,z))=xy+\alpha^{q+1}+z^2.$$
The polar form $B$ of $Q$ is given by $$B((x,y,\alpha,z),(x',y',\alpha',z'))=xy'+x'y+\alpha\alpha'^{q}+\alpha^{q}\alpha'+2zz'.$$

It is easy to check that the quadratic form $Q$ defined as above is non-degenerate and the associated quadric is a parabolic quadric $Q(4,q)$. Therefore, we can define $Q(4,q)$ as
\[Q(4,q)=\{\langle(x,y,\alpha,z)\rangle|0\neq(x,y,\alpha,z)\in V,\,Q((x,y,\alpha,z))=0\}.\]
In the remaining part of this paper, we will use $(x,y,\alpha,z)$ instead of $\langle(x,y,\alpha,z)\rangle$ to denote a projective point of $\text{PG}(4,q)$ for simplicity.

We choose this model such that the prescribed automorphism group $G$ we introduce now has a good presentation.

For each $\lambda\in\mathbb{F}_{q}^{*}$ and $\mu\in\mathbb{F}_{q^2}^{*}$ with $\mu^{\frac{q+1}{2}}=1$, we define
$$T_{\lambda,\mu}:(x,y,\alpha,z)\mapsto(x\lambda,y\lambda^{-1},\alpha\lambda^{\frac{q-1}{2}}\mu,z),$$
which is an isometry of $Q(4,q)$. Let $H$ be the group generated by all such $T_{\lambda,\mu}$'s, i.e., $H=\langle T_{\lambda,\mu}:\lambda\in\mathbb{F}^{*}_{q},\mu\in\mathbb{F}_{q^2}^{*},\mu^{\frac{q+1}{2}}=1\rangle$.
\begin{lemma}
   The group $H$ defined above is a cyclic subgroup of order $\frac{q^2-1}{2}$ of $\text{PGO}(5,q)$.
   \label{lem1}
  \end{lemma}
  \begin{proof}
  It is clear that $H$ is the direct product of the two cyclic subgroups $\langle T_{\lambda,1}:\lambda\in\mathbb{F}_{q}^{*}\rangle$ and $\langle T_{1,\mu}: \mu\in\mathbb{F}_{q^2}^{*},\mu^{\frac{q+1}{2}}=1\rangle$, which has order $q-1$ and $\frac{q+1}{2}$ respectively. Since $\text{gcd}(q-1,\frac{q+1}{2})=1$ for $q\equiv 1 (\text{mod}\ 4)$, the claim follows.
  \end{proof}
Furthermore, we define an involution $\tau$ in $\text{PGO}(5,q)$ as follows:
$$\tau:(x,y,\alpha,z)\mapsto(y,x,\alpha^{q},z),$$
which is also an isometry of $Q(4,q)$. Set $G:=\langle H,\tau\rangle$, which is isomorphic to $C_{\frac{q^2-1}{2}}\rtimes C_{2}$, where $C_{k}$ denotes a cyclic group of order $k$. This is the prescribed automorphism group for our putative $m$-ovoids in $Q(4,q)$.

\section{The $G$-orbits}

We now describe the $G$-orbits of $Q(4,q)$ with $G$ as defined above. For a point $P\in \text{PG}(4,q)$, let $O(P)$ denote the $G$-orbit containing $P$.

Let $\gamma$ be a fixed element of $\mathbb{F}_{q^2}$ with $\gamma^{q+1}=-1$, and let $\square_q$ (resp. $\square_{q^2}$) and $\blacksquare_q$ (resp. $\blacksquare_{q^2}$) be the set of nonzero squares of $\mathbb{F}_q$ (resp. $\mathbb{F}_{q^2}$) and nonsquares of $\mathbb{F}_q$ (resp. $\mathbb{F}_{q^2}$) respectively. All the $G$-orbits have length $\frac{q^2-1}{2}$ or $q^2-1$ with the following exceptions:
\begin{enumerate}
\item a unique orbit $O(1,0,0,0)=\{(1,0,0,0),(0,1,0,0)\}$ of length 2;
\item a unique orbit $O(1,-1,0,1)=C\setminus O(1,0,0,0)$ of length $q-1$ ,where $C=\{(x,y,0,z)\in V|xy+z^2=0, (x,y,0,z)\neq (0,0,0,0)\}$ is a conic;
\item a unique orbit $O(0,0,\gamma,1)$ of length $q+1$.
\end{enumerate}
We call the orbits of length $\frac{q^2-1}{2}$ \textit{short} orbits, the orbits of length $q^2-1$ \textit{long} orbits and the remaining orbits \textit{exceptional}.
\begin{lemma}
   Let $x,y\in\mathbb{F}^{*}_{q}$, $\alpha\in\mathbb{F}^{*}_{q^2}$, and $(x,y,\alpha,1)\in Q(4,q)$. Then the $G$-orbit $O(x,y,\alpha,1)$ is a short orbit if and only if $xy(1+xy)\in\square_q$, or equivalently $\alpha^{q+1}(1+\alpha^{q+1})\in\square_{q}$.
   \label{lem2}
  \end{lemma}
\begin{proof} $O(x,y,\alpha,1)$ is a short orbit if and only if its stabilizer in $G$ has order 2 since the order of $G$ is $q^2-1$. If $T_{\lambda,\mu}$ stabilizes $(x,y,\alpha,1)$, then there exists a constant $c\in\mathbb{F}^{*}_{q}$ such that $(x\lambda,y\lambda^{-1},\alpha\lambda^{\frac{q-1}{2}}\mu,1)=c\cdot(x,y,\alpha,1)$. It follows that $c=1=\lambda=\lambda^{-1}=\lambda^{\frac{q-1}{2}}\mu$, which gives $\lambda=\mu=1$. If $T_{\lambda,\mu}\cdot\tau$ stabilizes $(x,y,\alpha,1)$, then there exists a constant $c'\in\mathbb{F}_{q}^{*}$ such that $(y\lambda,x\lambda^{-1},\alpha^{q}\lambda^{\frac{q-1}{2}}\mu,1)=c'\cdot(x,y,\alpha,1)$. We get $c'=1$, $\lambda=xy^{-1}$, $\mu=(yx^{-1})^{\frac{q-1}{2}}\alpha^{1-q}$, which indicates that the value of $\lambda,\mu$ is uniquely determined by $x,y,\alpha$ respectively. In order for $T_{\lambda,\mu}\cdot\tau$ to be in $G$, we need $\mu^{\frac{q+1}{2}}=1$, i.e. $(yx^{-1})^{\frac{q-1}{2}}\alpha^{\frac{1-q^2}{2}}=1$, which is equivalent to $xy\alpha^{q+1}$ is a square of $\mathbb{F}_{q}^{*}$. The claim now follows since $xy+\alpha^{q+1}+1=0$ and $-1\in\square_{q}$ for $q\equiv 1 (\text{mod}\ 4)$.
  \end{proof}

Let $\omega$ be a fixed primitive element of $\mathbb{F}_{q^2}$ and let $\gamma$ be a fixed element of $\mathbb{F}_{q^2}$ with $\gamma^{q+1}=-1$ as introduced above. According to Lemma \ref{lem2}, we are now ready to give an explicit description of short and long orbits below. For the size of $S\cap (S'-1)$, where $S,S'\in\{\square_{q},\,\blacksquare_{q}\}$, please refer to Remark \ref{rem1} in the next section. Here $S'-1=\{x-1|\,x\in S'\}$.

There are a total of $q-1$ short orbits of length $\frac{q^2-1}{2}$, which we list below.
\begin{enumerate}
  \item The point set $\{(x,y,\alpha,0)\in Q(4,q)|\,\alpha\neq 0\}$ splits into two orbits, $O(1,-\omega^{2(q+1)},\omega^2,0)$ and $O(1,-\omega^{q+1},\omega,0)$. Both orbits have size $\frac{q^2-1}{2}$.

  \item For each $y\in\square_{q}\cap(\square_{q}-1)$, there are two orbits $O(1,y,\alpha,1)$ and $O(1,y,\alpha^{q},1)$ with $y+\alpha^{q+1}+1=0$. In total, there are $2\cdot|\square_{q}\cap(\square_{q}-1)|=\frac{q-5}{2}$ such orbits of length $\frac{q^2-1}{2}$.

  \item For each $y\in\blacksquare_{q}\cap(\blacksquare_{q}-1)$, there are two orbits $O(1,y,\alpha,1)$ and $O(1,y,\alpha^{q},1)$ with $y+\alpha^{q+1}+1=0$. In total, there are $2\cdot|\blacksquare_{q}\cap(\blacksquare_{q}-1)|=\frac{q-1}{2}$ such orbits of length $\frac{q^2-1}{2}$.
\end{enumerate}

There are a total of $\frac{q+3}{2}$ long orbits of length $q^2-1$, which we list below.
\begin{enumerate}
\item There are two orbits of length $q^2-1$ with points whose second coordinate zero, namely, $O(1,0,\gamma,1)$ and $O(1,0,\gamma\omega^{q-1},1)$.

\item For each $y\in\square_{q}\cap(\blacksquare_{q}-1)$, there is exactly one orbit $O(1,y,\alpha,1)$ with $y+\alpha^{q+1}+1=0$. In total, there are $|\square_{q}\cap(\blacksquare_{q}-1)|=\frac{q-1}{4}$ such orbits of length $q^2-1$.

\item For each $y\in\blacksquare_q\cap(\square_q-1)$, there is exactly one orbit $O(1,y,\alpha,1)$ with $y+\alpha^{q+1}+1=0$. In total, there are $|\blacksquare_{q}\cap(\square_{q}-1)|=\frac{q-1}{4}$ such orbits of length $q^2-1$.
\end{enumerate}

In particular, there are 7 orbits in which a representative has a coordinate being zero: $O(1,0,0,0)$, $O(1,-1,0,1)$, $O(0,0,\omega^{\frac{q-1}{2}},1)$, $O(1,-\omega^{2(q+1)},\omega^2,0)$, $O(1,-\omega^{q+1},\omega,0)$, $O(1,0,\omega^{\frac{q-1}{2}},1)$ and $O(1,0,\omega^{\frac{3(q-1)}{2}},1)$.

\section{The construction of $\frac{q-1}{2}$-ovoids in $Q(4,q)$}

We are now ready to describe the construction of $\frac{q-1}{2}$-ovoids in $Q(4,q)$. Let $\omega$ be a primitive element of $\mathbb{F}_{q^2}^{*}$ as introduced above. Fix a pair $(a,b)$ in $\mathbb{F}_{q}^{*}$ such that
\begin{equation}
1+a^2=b^2.
\label{1paseqba}
\end{equation}
Now we define
\begin{equation}\begin{array}{rl}
\mathcal{M}= &O(1,-1,0,1)\cup O(1,-\omega^{2(q+1)},\omega^2,0)\cup\\
   &O(1,0,\omega^{\frac{q-1}{2}},1) \cup \mathcal{T}\cup  O(1,-b^2,a,1).
\label{eq:1}
\end{array}
\end{equation}
where
\begin{equation}
\mathcal{T}=\{(x,y,\alpha,1)\in Q(4,q)\ |\ 1+b^{-2}xy\in\square_{q},\ xy\alpha\neq0\ \}.
\label{mathcalT}
\end{equation}

\begin{lemma}
The set $\mathcal{T}$ is $G$-invariant with $\frac{(q^2-1)(q-5)}{2}$ points.
\label{lemt}
\end{lemma}

\begin{proof}
It is straightforward to check that $\mathcal{T}$ is $G$-invariant. It remains to compute the size of $\mathcal{T}$. We have
\begin{align*}
|\mathcal{T}|
&=(q-1)|\{z\in\mathbb{F}_{q}^{*},\,\alpha\in\mathbb{F}_{q^2}^{*}|\,z+1+\alpha^{q+1}=0,\,1+b^{-2}z\in\square_{q}\}|\\ &=(q^2-1)|\{z\in\mathbb{F}_{q}|\,1+b^{-2}z\in\square_{q},\,z\neq0,-1\}|\\
&=(q^2-1)(|\{z\in\mathbb{F}_{q}|\, 1+b^{-2}z\in \square_{q}\}|-2)\\
&=(q^2-1)\left(\frac{q-1}{2}-2\right)=\frac{(q^2-1)(q-5)}{2}.
\end{align*}
 Here, we have made use of the fact that $1-b^{-2}=a^2b^{-2}$ is a square in the third equality and $|b^2(\square_{q}-1)|=\frac{q-1}{2}$ in the fourth equality. This proof is complete.
\end{proof}

By Lemma \ref{lemt}, we have $$|\mathcal{M}|=(q-1)+2(q^2-1)+\frac{(q^2-1)(q-5)}{2}=\frac{q-1}{2}(q^2+1),$$ which is exactly the size of a $\frac{q-1}{2}$-ovoid in $Q(4,q)$.
Let $\eta$ be the quadratic (multiplicative) character of $\mathbb{F}_{q}$, i.e.,
\begin{equation}
\eta(x)=\left\{\begin{array}{lll}1 & \text{if\  } x\in\square_{q},\\-1 &  \text{if\  } x\in\blacksquare_{q}, \\0 & \text{if\ } x=0. \end{array}\right.
\label{defeta}
\end{equation}
Furthermore, we define the Kronecker delta function $[[\mathcal{P}]]$ as follows
\begin{equation}
[[\mathcal{P}]]=\left\{\begin{array}{ll}1, & \text{if the property\ }\mathcal{P} \text{\ holds},\\0, &  \text{otherwise. } \end{array}\right.
\label{Kronecker}
\end{equation}
\begin{lemma}
\cite[Theorem 5.48]{lidl1997finite} Let $g(x)=a_{2}x^{2}+a_{1}x+a_{0}\in\mathbb{F}_{q}[x]$ with $q$ odd and $a_{2}\neq0$. Set $d=a_{1}^{2}-4a_{0}a_{2}$ and let $\eta$ be the quadratic character of $\mathbb{F}_{q}$. Then $$\sum_{c\in\mathbb{F}_{q}}\eta(g(c))=\left\{\begin{array}{ll}-\eta(a_{2}) & \text{if\  } d\neq 0,\\(q-1)\eta(a_{2}) &  \text{if\  } d=0. \end{array}\right.$$
\label{lem6}
\end{lemma}

 \begin{remark}
 \label{rem1}
 Consider the special case $g(x)=x(x+1)$, i.e. $a_1=a_2=1$, $a_0=0$. We have $d=a_{1}^{2}-4a_{0}a_{2}=1$, so $\sum_{c\in\mathbb{F}_{q}}\eta(g(c))=-1$. Let $n_{1},\,n_{2},\,n_{3},\,n_{4}$ be the number of $x$ such that $(\eta(x),\eta(x+1))=$ $(1,1)$, $(1,-1)$, $(-1,1)$, $(-1,-1)$ respectively. Then
 \begin{align*}
 &n_{1}+n_{2}=\frac{q-1}{2}-1,\ n_{3}+n_{4}=\frac{q-1}{2},\\ &n_{1}+n_{3}=\frac{q-1}{2}-1,\ n_{2}+n_{4}=\frac{q-1}{2}.
 \end{align*}
 From $\sum_{c\in\mathbb{F}_{q}}\eta(g(c))=-1$, we get $n_{1}-n_{2}-n_{3}+n_{4}=-1$. We solve from these equations that \[n_{1}=\frac{q-5}{4},\ n_{2}=\frac{q-1}{4},\ n_{3}=\frac{q-1}{4},\ n_{4}=\frac{q-1}{4}.\] In particular, $n_{1}$ is the size of $\square_{q}\cap(\square_{q}-1)$, $n_{2}$ is the size of $\square_{q}\cap(\blacksquare_{q}-1)$, $n_{3}$ is the size of $\blacksquare_{q}\cap(\square_{q}-1)$, and $n_{4}$ is the size of $\blacksquare_{q}\cap(\blacksquare_{q}-1)$ by definition.
\end{remark}
\begin{thm} The point set $\mathcal{M}$ in (\ref{eq:1}) is a $\frac{q-1}{2}$-ovoid of $Q(4,q)$ for $q\equiv 1 (\text{mod}\ 4)$ and $q>5$.
\end{thm}
\begin{proof}
We take the same technique as in \cite{feng2016family}, i.e., we show that each line of $Q(4,q)$ meets $\mathcal{M}$ in $\frac{q-1}{2}$ points. Each line of $Q(4,q)$ intersects the hyperplane $\{(x,y,\alpha,z)\in V|\;y=0\}$ in at least one point. There are four $G$-orbits of $Q(4,q)$ that have a representative with second coordinate zero, namely, $O(1,0,0,0)$, $O(0,0,\omega^{\frac{q-1}{2}},1)$, $O(1,0,\omega^{\frac{q-1}{2}},1)$ and $O(1,0,\omega^{\frac{3(q-1)}{2}},1)$. Since $\mathcal{M}$ is $G$-invariant, we only need to consider the lines through $(1,0,0,0)$, $(0,0,\omega^{\frac{q-1}{2}},1)$, $(1,0,\omega^{\frac{q-1}{2}},1)$ or $(1,0,\omega^{\frac{3(q-1)}{2}},1)$. By the assumption of $q>5$ and Lemma \ref{lemt}, the set $\mathcal{T}$ is non-empty.\\

\textbf{Case 1.} The line $\ell$ of $Q(4,q)$ passes through $P=(1,0,0,0)$.

The line $\ell$ intersects the hyperplane $\{(x,y,\alpha,z)\in V:\,x=0\}$ in exactly one point $Q$, perpendicular to $P$, with $Q=(0,y_{1},\alpha_{1},z_{1})$. Since $B(P,Q)=0$, we have $y_{1}=0$ and $\alpha_{1}^{q+1}+z_{1}^2=0$, $z_{1}\neq0$ by the definition of $B$ and $Q(4,q)$. So we can set $z_{1}=1$. Therefore, $\ell=\langle P,Q\rangle$ with $Q=(0,0,\alpha_{1},1)$ for some $\alpha_{1}\in\mathbb{F}_{q^2}^{*}$ with $\alpha_{1}^{q+1}+1=0$. The line $\ell$ can be denoted as $\ell=\{(t,0,\alpha_{1},1)|t\in\mathbb{F}_{q}\}\cup\{(1,0,0,0)\}$. It is clear that $P=(1,0,0,0)\notin\mathcal{M}$ in this case. For any $(x,y,\alpha,z)\in$ $O(1,-1,0,1)$, $O(1,-\omega^{2(q+1)},\omega^2,0)$, $\mathcal{T}$ or $O(1,-b^2,a,1)$, we have $y\neq 0$, so it can not lie on the line $\ell$. Therefore, $|\ell\cap\mathcal{M}|=|\ell\cap O(1,0,\omega^{\frac{q-1}{2}},1)|$. We now show that this size is $\frac{q-1}{2}$.

The orbit $O(1,0,\omega^{\frac{q-1}{2}},1)$ is a long orbit with $q^2-1$ elements. It is the union of
\begin{equation}
U_{1}=\{(\lambda,0,\lambda^{\frac{q-1}{2}}\mu\omega^{\frac{q-1}{2}},1)|\lambda\in\mathbb{F}_{q}^{*},\mu\in\mathbb{F}_{q^2}^{*},\mu^{\frac{q+1}{2}}=1\},
\label{eq:3}
\end{equation}
 and
\begin{equation}
 U_{2}=\{(0,\lambda,-\lambda^{\frac{q-1}{2}}\mu^{-1}\omega^{-\frac{q-1}{2}},1)|\lambda\in\mathbb{F}_{q}^{*},\mu\in\mathbb{F}_{q^2}^{*},\mu^{\frac{q+1}{2}}=1\}.
\label{eq:4}
\end{equation}
 It is clear that $U_{2}\cap\ell=\emptyset$ by examining the second coordinate.
Suppose that the point $(t,0,\alpha_{1},1)$ of $\ell$ lies in $U_{1}$, where $t\in\mathbb{F}_{q}$. Then there exists
$\lambda\in\mathbb{F}_{q}^{*}$, $\mu\in\mathbb{F}_{q^2}^{*}$ with $\mu^{\frac{q+1}{2}}=1$ such that
\[(t,0,\alpha_{1},1)=c\cdot(\lambda,0,\lambda^{\frac{q-1}{2}}\mu\omega^{\frac{q-1}{2}},1)\]
for some $c\in\mathbb{F}_{q}^{*}$. The last coordinate gives that $c=1$ and then the first coordinate gives that $\lambda=t$, which means $t\neq0$. By comparing third coordinate, we get
\[\mu=\alpha_{1}t^{-\frac{q-1}{2}}\omega^{-\frac{q-1}{2}}\]
and
\[\mu^{\frac{q+1}{2}}=\alpha_{1}^{\frac{q+1}{2}}t^{-\frac{q-1}{2}}\omega^{-\frac{q^2-1}{4}}=1.\]
Therefore,
 \begin{align*}
 |\ell\cap U_{1}| &=\#\{t\in\mathbb{F}_{q}^{*}|\,t^{\frac{q-1}{2}}=\alpha_{1}^{\frac{q+1}{2}}\omega^{-\frac{q^2-1}{4}}\}\\
 &=\frac{q-1}{2}.
\end{align*}
The last equality holds since \[\left(\alpha_{1}^{\frac{q+1}{2}}\omega^{-\frac{q^2-1}{4}}\right)^{2}=(\alpha_{1}^{q+1})(\omega^{-\frac{q^2-1}{2}})=1\]
by $\alpha_{1}^{q+1}+1=0$ and then $\alpha_{1}^{\frac{q+1}{2}}\omega^{-\frac{q^2-1}{4}}\in\{1,\,-1\}$.\\

\textbf{Case 2.} The line $\ell$ of $Q(4,q)$ passes through $P=(0,0,\omega^{\frac{q-1}{2}},1)$.

The line $\ell$ intersects the hyperplane $\{(x,y,\alpha,z)\in V|\,z=0\}$ in exactly one point $Q=(x_{1},y_{1},\alpha_{1},0)$, where $x_{1}y_{1}+\alpha_{1}^{q+1}=0$. If $\alpha_{1}=0$, then $x_{1}y_{1}=0$. It means $Q=(1,0,0,0)$ or $(0,1,0,0)$, which lies in the same $G$-orbit, and we are reduced to Case 1. We assume that $\alpha_{1}\neq0$ below. From $P\perp Q$, we get $\text{Tr}_{\mathbb{F}_{q^2}/\mathbb{F}_{q}}(\alpha_{1}\omega^{\frac{q^2-q}{2}})=0$, i.e. $\alpha_{1}=\omega^{-1}a_{1}$ for some $a_{1}\in\mathbb{F}_{q}^{*}$. We can assume that the point $Q$ equals $(1,y_{1},\alpha_{1},0)$ with $y_{1}+\alpha_{1}^{q+1}=0$ since $x_{1}y_{1}\neq0$ and the line \[\ell=\{(1,y_{1},\alpha_{1}+t\omega^{\frac{q-1}{2}},t)|t\in\mathbb{F}_{q}\}\cup\{(0,0,\omega^{\frac{q-1}{2}},1)\}.\]
Obviously, the point $P=(0,0,\omega^{\frac{q-1}{2}},1)\notin\mathcal{M}$ in this case. It is clear that $\alpha_{1}=\omega^{-1}a_{1}$ is a nonsquare of $\mathbb{F}_{q^2}^{*}$. Below we calculate the intersection number of $\ell$ with each part of $\mathcal{M}$ respectively.

\begin{enumerate}
\item[(2.1)] $|\ell\cap O(1,-1,0,1)|=0$.\\
  Suppose that the point $(1,y_{1},\alpha_{1}+t\omega^{\frac{q-1}{2}},t)$ of $\ell$ lies in $O(1,-1,0,1)$, where $t\in\mathbb{F}_{q}$. Then its third coordinate must be 0, i.e., $\alpha_{1}+t\omega^{\frac{q-1}{2}}=0$. We have shown that $\alpha_{1}$ is a nonsquare of $\mathbb{F}_{q^2}^{*}$ above. On the other hand, $-\omega^{\frac{q-1}{2}}t$ is a square or 0 of $\mathbb{F}_{q}$ since $q\equiv 1(\text{mod}\ 4)$. This contradiction completes the proof.


\item[(2.2)] $|\ell\cap O(1,-\omega^{2(q+1)},\omega^2,0)|=0$.\\
Each element in this intersection has a zero last coordinate and the only element in $\ell$ with this property is $Q$. On the other hand, $Q$ is not in $O(1,-\omega^{2(q+1)},\omega^2,0)=\{(\lambda,-\omega^{2(q+1)}\lambda^{-1},\omega^2\lambda^{\frac{q-1}{2}}\mu,0)|\lambda\in\mathbb{F}_{q}^{*},\mu\in\mathbb{F}_{q^2}^{*},\mu^{\frac{q+1}{2}}=1\}$ since otherwise there exists $\lambda\in\mathbb{F}_{q}^{*}$ such that $y_{1}=-\omega^{2(q+1)}\lambda^{-2}$ by comparing the second coordinate, and we get a contradiction since $y_{1}=-\alpha_{1}^{q+1}$ is a nonsquare of $\mathbb{F}_{q}$.

\item[(2.3)] $|\ell\cap O(1,0,\omega^{\frac{q-1}{2}},1)|=0$.\\
Each element in this intersection has a zero second coordinate and the only element in $\ell$ with this property is $P$ since $y_{1}\neq0$. However, $P=(0,0,\omega^{\frac{q-1}{2}},1)\notin O(1,0,\omega^{\frac{q-1}{2}},1)$ because $O(0,0,\omega^{\frac{q-1}{2}},1)$ and $O(1,0,\omega^{\frac{q-1}{2}},1)$ are different $G$-orbits.

\item[(2.4)] $|\ell\cap\mathcal{T}|=\frac{q-1}{2}$.\\
Suppose that the point $(1,y_{1},\alpha_{1}+t\omega^{\frac{q-1}{2}},t)$ of $\ell$ lies in $\mathcal{T}$, where $t\in\mathbb{F}_{q}$. Then there exists
$(x,y,\alpha,1)\in\mathcal{T}$ such that
\[(1,y_1,\alpha_{1}+t\omega^{\frac{q-1}{2}},t)=c\cdot(x,y,\alpha,1)\]
for some $c\in\mathbb{F}_{q}^{*}$. The last coordinate gives that $c=t$. In particular, $t\neq0$. By comparing the other coordinates, we get
 \[1=tx,\,y_{1}=ty,\,\alpha_{1}+t\omega^{\frac{q-1}{2}}=t\alpha.\]
It follows that $x=t^{-1}$, $y=t^{-1}y_{1}$ and $\alpha=t^{-1}\alpha_{1}+\omega^{\frac{q-1}{2}}$. Therefore,
 \begin{align*}
|\ell\cap\mathcal{T}|=\#\{t\in\mathbb{F}_{q}^{*}|1+(tb)^{-2}y_{1}\in\square_{q}\}
 \end{align*}
by the definition of $\mathcal{T}$ in Eqn. (\ref{mathcalT}). As we mentioned above, $y_{1}$ is a nonsquare in $\mathbb{F}_{q}^{*}$. Hence
 \begin{align*}
|\ell\cap\mathcal{T}|&=2\cdot|b^{-2}y_{1}\square_{q}\cap (\square_{q}-1)|\\
&=2\cdot|\blacksquare_{q}\cap(\square_{q}-1)|=\frac{q-1}{2}.
 \end{align*}

\item[(2.5)] $|\ell\cap O(1,-b^{2},a,1)|=0$.\\
The orbit $O(1,-b^{2},a,1)$ is a short orbit with
\begin{equation}
O(1,-b^{2},a,1)=\{(\lambda,-b^2\lambda^{-1},a\lambda^{\frac{q-1}{2}}\mu,1):\lambda\in\mathbb{F}_{q}^{*},\mu\in\mathbb{F}_{q^2}^{*},\mu^{\frac{q+1}{2}}=1\}.
\label{express:Oba}
\end{equation}
Suppose that the point $(1,y_1,\alpha_{1}+t\omega^{\frac{q-1}{2}},t)$ of $\ell$ lies in $O(1,-b^{2},a,1)$, where $t\in\mathbb{F}_{q}$. Then by comparing the coordinates, we get $y_{1}=-t^2b^2$, a contradiction to the fact that $y_{1}$ is a nonsquare of $\mathbb{F}_{q}$.
\end{enumerate}
To sum up, we get $|\ell\cap\mathcal{M}|=\frac{q-1}{2}$. This completes the proof for Case 2.\\

\textbf{Case 3.} The line $\ell$ of $Q(4,q)$ passes through $P=(1,0,\omega^{\frac{q-1}{2}},1)$.

Similar to Case 2, we only need to consider the case that $\ell$ passes through a point $Q=(1,-\alpha_{1}^{q+1},\alpha_{1},0)$ for some $\alpha_{1}\in\mathbb{F}_{q^2}^{*}$. From $P\perp Q$, we deduce that $-\alpha_{1}^{q+1}+\text{Tr}_{\mathbb{F}_{q^2}/\mathbb{F}_{q}}(\alpha_{1}\omega^{\frac{q^2-q}{2}})=0$. Set
\[y_{1}=-\alpha_{1}^{q+1},\quad \beta=\alpha_{1}\omega^{\frac{-(q-1)}{2}}.\]
Then
\begin{equation}
\beta^{q+1}=(\alpha_{1}\omega^{\frac{-(q-1)}{2}})^{q+1}=-\alpha_{1}^{q+1}=y_{1},
\label{y1eqbetaqp1}
\end{equation}
and
\begin{equation}
\text{Tr}_{\mathbb{F}_{q^2}/\mathbb{F}_{q}}(\beta)=-\text{Tr}_{\mathbb{F}_{q^2}/\mathbb{F}_{q}}(\alpha_{1}\omega^{\frac{q^2-q}{2}})=-\alpha_{1}^{q+1}=y_{1}.
\label{y1eqTrbeta}
\end{equation}
In particular, $\beta^{q+1}=\beta+\beta^{q}$ with $\beta\in\mathbb{F}_{q^2}^{*}$, which is equivalent to
\begin{equation}
\beta^{q}=\beta^{q-1}+1
\label{betaqeqbetaqm1p1}
\end{equation}
and
\begin{equation}
\beta-1=\beta^{-(q-1)}.
\label{betam1eqbetaqm1}
\end{equation}
 We claim that $\beta\in\mathbb{F}_{q}$ if and only if $\beta=2$ : if $\beta\in\mathbb{F}_{q}$, then $\beta=\beta^{q}=\beta^{q-1}+1=2$ by Eqn. (\ref{betaqeqbetaqm1p1}); the converse is obvious. The proof for the case $\beta=2$ and $\beta\neq 2$ are basically the same and the former is easier, so we will only prove the case $\beta\neq2$ below. In this case, we know that the minimal polynomial of $\beta$ over $\mathbb{F}_{q}$ is $X^{2}-y_{1}X+y_{1}$ by Eqn. (\ref{y1eqbetaqp1}) and Eqn. (\ref{y1eqTrbeta}), and its discriminant $d'=y_{1}(y_{1}-4)$ is a nonsquare of $\mathbb{F}_{q}$ since $\beta\notin\mathbb{F}_{q}$.

With all the preparations, we are now ready to compute the intersection size of $\ell$ with each part of $\mathcal{M}$. We have
\[\ell=\{(t+1,y_1,(\beta+t)\omega^{\frac{q-1}{2}},t)|t\in\mathbb{F}_{q}\}\cup\{(1,0,\omega^{\frac{q-1}{2}},1)\}\]
with $y_1\neq 4$ (i.e. $\beta\neq 2$) and $\beta=\alpha_{1}\omega^{\frac{-(q-1)}{2}}$. It is clear that $P=(1,0,\omega^{\frac{q-1}{2}},1)\in O(1,0,\omega^{\frac{q-1}{2}},1)\subset\mathcal{M}$ in this case.

\begin{enumerate}

\item[(3.1)] $|\ell\cap O(1,-1,0,1)|=0$.\\
Suppose that the point $(t+1,y_1,(\beta+t)\omega^{\frac{q-1}{2}},t)$ of $\ell$ lies in $O(1,-1,0,1)$, where $t\in\mathbb{F}_{q}$. Comparing the third coordinate gives that $\beta=-t\in\mathbb{F}_{q}$, i.e., $y_1=4$. This is a contradiction.

\item[(3.2)] $|\ell\cap O(1,-\omega^{2(q+1)},\omega^2,0)|=[[y_{1}\in\square_{q}]]$.\\
Suppose that the point $(t+1,y_1,(\beta+t)\omega^{\frac{q-1}{2}},t)$ of $\ell$ lies in the G-orbit $O(1,-\omega^{2(q+1)},\omega^2,0)$. The last coordinate must be zero, then $t=0$. We have the point $(1,y_{1},\beta\omega^{\frac{q-1}{2}},0)\in\ell\cap O(1,-\omega^{2(q+1)},\omega^2,0)$ if and only if $y_{1}$ is a nonzero square of $\mathbb{F}_{q}$ by comparing the coordinates. Therefore, this claim follows by the definition of the Kronecker delta function in Eqn. (\ref{Kronecker}).

\item[(3.3)] $|\ell\cap O(1,0,\omega^{\frac{q-1}{2}},1)|=1$.\\
On one hand, $P$ is a common point of $\ell$ and $O(1,0,\omega^{\frac{q-1}{2}},1)$. On the other hand, suppose that there is a point $(t+1,y_1,(\beta+t)\omega^{\frac{q-1}{2}},t)$ of $\ell$ that lies in $O(1,0,\omega^{\frac{q-1}{2}},1)$ with $y_{1}\neq 0$. Since $y_{1}\neq 0$, this point must be in $U_{2}$ by Eqn. (\ref{eq:3}) and Eqn. (\ref{eq:4}). So we obtain that $t=-1$. There exist $\lambda\in\mathbb{F}_{q}^{*}$, $\mu\in\mathbb{F}_{q^2}^{*}$ with $\mu^{\frac{q+1}{2}}=1$ such that \[(0,y_{1},(\beta-1)\omega^{\frac{q-1}{2}},-1)=c\cdot(0,\lambda,-\lambda^{\frac{q-1}{2}}\mu^{-1}\omega^{-\frac{q-1}{2}},1)\] for some $c\in\mathbb{F}_{q}^{*}$. It follows that $c=-1$, $\lambda=-y_{1}$,
\[\mu=\lambda^{\frac{q-1}{2}}(\beta-1)^{-1}\omega^{-(q-1)},\]
and
\[\mu^{\frac{q+1}{2}}=-\lambda^{\frac{q-1}{2}}(\beta-1)^{-\frac{q+1}{2}}=1.\]
Then we get
\begin{equation}
y_{1}^{\frac{q-1}{2}}=(-\lambda)^{\frac{q-1}{2}}=-(\beta-1)^{\frac{q+1}{2}}=-(\beta^q-1)^{\frac{q^2+q}{2}}.
\label{y1betaanre}
\end{equation}
We deduce that $y_{1}^{\frac{q-1}{2}}=-\beta^{\frac{q^2-1}{2}}$ by Eqn. (\ref{betaqeqbetaqm1p1}) and Eqn. (\ref{y1betaanre}). On the other hand, $y_{1}=\beta^{q+1}$ by Eqn. (\ref{y1eqbetaqp1}), which gives $y_{1}^{\frac{q-1}{2}}=\beta^{\frac{q^2-1}{2}}$. This is a contradiction.

\item[(3.4)] $|\ell\cap\mathcal{T}|=\frac{q-3}{2}-\frac{1}{2}(\eta(y_{1}(4b^2-y_{1}))+1)-[[y_{1}\in\square_{q}]]$.\\
Suppose that the point $(t+1,y_1,(\beta+t)\omega^{\frac{q-1}{2}},t)$ of $\ell$ lies in $\mathcal{T}$, where $t\in\mathbb{F}_{q}$. Then there exists
$(x,y,\alpha,1)\in\mathcal{T}$ such that
\[(t+1,y_1,(\beta+t)\omega^{\frac{q-1}{2}},t)=c\cdot(x,y,\alpha,1)\]
for some $c\in\mathbb{F}_{q}^{*}$. The last coordinate gives that $c=t$. In particular, $t\neq0$ and $t\neq-1$ since $t+1=tx$, $tx\neq0$ and $t=c$. By comparing the other coordinates, we get
 \[t+1=tx,\,y_{1}=ty,\,(\beta+t)\omega^{\frac{q-1}{2}}=t\alpha.\]
It follows that $x=1+t^{-1}$, $y=t^{-1}y_{1}$ and $\alpha=t^{-1}(\beta+t)\omega^{\frac{q-1}{2}}$. Therefore,
 \begin{align*}
|\ell\cap\mathcal{T}|&=\#\{t\in\mathbb{F}_{q}|\,1+b^{-2}(1+t^{-1})t^{-1}y_{1}\in\square_{q},\,t\neq 0,-1\}\\
&=\#\{t\in\mathbb{F}_{q}|\,b^{2}t^2+y_{1}t+y_{1}\in\square_{q},\,t\neq 0,-1\}.
 \end{align*}
Set $g(X)=b^2X^2+y_{1}X+y_{1}$. We have $g(0)=y_{1}$ and $g(-1)=b^2\in\square_{q}$. Let $\eta$ be the quadratic character of $\mathbb{F}_{q}$ as introduced in Eqn. (\ref{defeta}). The discriminant of $g(X)$ is nonzero, i.e., $y_{1}^{2}-4b^2y_{1}\neq0$. Otherwise, $y_{1}=4b^2$ and $d'=y_{1}(y_{1}-4)=4b^2(4(b^2-1))=16b^2a^2\in\square_{q}$, a contradiction. Hence $g(X)=0$ has 2 or 0 solutions in $\mathbb{F}_{q}$ according as the discriminant $y_{1}^{2}-4b^2y_{1}$ is a square or not. In other words, the number of solutions to $g(X)=0$ in $\mathbb{F}_{q}$ equals $\eta(y_{1}(4b^2-y_{1}))+1$. Then we obtain that
 \begin{align*}
&|\ell\cap\mathcal{T}|=\frac{1}{2}\sum_{t\in\mathbb{F}_{q}, b^2t^2+y_{1}t+y_{1}\neq0}\left(\eta(b^2t^2+y_{1}t+y_{1})+1\right)-[[y_{1}\in\square_{q}]]-1\\
&=\frac{1}{2}\sum_{t\in\mathbb{F}_{q}}\left(\eta(b^2t^2+y_{1}t+y_{1})+1\right)-\frac{1}{2}|\{t\in\mathbb{F}_{q}|\,b^2t^2+y_{1}t+y_{1}=0\}|-[[y_{1}\in\square_{q}]]-1\\
&=\frac{q-2}{2}+\frac{1}{2}\sum_{t\in\mathbb{F}_{q}} \eta(b^2t^2+y_{1}t+y_{1})-\frac{1}{2}|\{t\in\mathbb{F}_{q}|\,b^2t^2+y_{1}t+y_{1}=0\}|-[[y_{1}\in\square_{q}]]\\
&=\frac{q-2}{2}+\frac{1}{2}(-\eta(b^2))-\frac{1}{2}|\{t\in\mathbb{F}_{q}|\,b^2t^2+y_{1}t+y_{1}=0\}|-[[y_{1}\in\square_{q}]]\\
&=\frac{q-3}{2}-\frac{1}{2}\left(\eta(y_{1}(4b^2-y_{1}))+1\right)-[[y_{1}\in\square_{q}]].
 \end{align*}
We get the fourth equality by $y_{1}^2-4b^2y_{1}\neq0$ and Lemma \ref{lem6}.
\item[(3.5)] $|\ell\cap O(1,-b^{2},a,1)|=\frac{1}{2}\left(\eta(y_{1}(4b^2-y_{1}))+1\right)$.\\
In the case of $\eta(y_{1}(4b^2-y_{1}))=-1$, i.e., $y_{1}(4b^2-y_{1})\in\blacksquare_{q}$, we show that $|\ell\cap O(1,-b^2,a,1)|=0$. Recall that the elements of $O(1,-b^{2},a,1)$ are listed in Eqn. (\ref{express:Oba}). Suppose that the point $(t+1,y_1,(\beta+t)\omega^{\frac{q-1}{2}},t)$ of $\ell$ lies in $O(1,-b^2,a,1)$, then there exists
$\lambda\in\mathbb{F}_{q}^{*}$, $\mu\in\mathbb{F}_{q^2}^{*}$ with $\mu^{\frac{q+1}{2}}=1$ such that
\[(t+1,y_1,(\beta+t)\omega^{\frac{q-1}{2}},t)=c\cdot(\lambda,-b^2\lambda^{-1},a\lambda^{\frac{q-1}{2}}\mu,1)\]
for some $c\in\mathbb{F}_{q}^{*}$. By comparing coordinates, we get
\[c=t,\,t+1=c\lambda,\,y_{1}=-cb^2\lambda^{-1},\,(\beta+t)\omega^{\frac{q-1}{2}}=ca\lambda^{\frac{q-1}{2}}\mu.\]
In particular, $t\neq0,-1$ since $c\neq0$ and $\lambda\neq0$. It follows that $\lambda=1+t^{-1}$. Plugging into the third equation above, we get
 \[y_{1}\lambda+cb^2=y_{1}(1+t^{-1})+tb^2=0,\ \textup{i.e.},\ b^2t^2+y_{1}t+y_{1}=0.\]
It means that $b^2X^2+y_{1}X+y_{1}=0$ has a solution in $\mathbb{F}_{q}$. However, the equation $b^2X^2+y_{1}X+y_{1}=0$ has discriminant $y_{1}(4b^2-y_{1})\in\blacksquare_{q}$ by assumption, so has no solution in $\mathbb{F}_{q}$, a contradiction.

Next, we consider the case $\eta(y_{1}(4b^2-y_{1}))=1$, i.e., $y_{1}(4b^2-y_{1})\in\square_{q}$, and we show that $|\ell\cap O(1,-b^2,a,1)|=1$. Suppose that the point $(t+1,y_1,(\beta+t)\omega^{\frac{q-1}{2}},t)$ of $\ell$ lies in $O(1,-b^2,a,1)$. By the same argument, we have $t\neq0,-1$ and there exists $\lambda\in\mathbb{F}_{q}^{*}$, $\mu\in\mathbb{F}_{q^2}^{*}$ with $\mu^{\frac{q+1}{2}}=1$ such that
\[\lambda=t^{-1}+1,\,b^2t^2+y_{1}t+y_{1}=0,\]
and
\[\mu=(\beta+t)\omega^{\frac{q-1}{2}}(at(t^{-1}+1)^{\frac{q-1}{2}})^{-1}.\]
In this case, $b^2X^2+y_{1}X+y_{1}=0$ has two distinct solutions $t_{1}$, $t_{2}$ in $\mathbb{F}_{q}$ such that
\begin{equation}
\left\{\begin{array}{l}
t_{1}+t_{2}=-y_{1}b^{-2},\\
t_{1}t_{2}=y_{1}b^{-2}.
\end{array}\right.
\label{eqt1t2}
\end{equation}
We deduce that
\[t_1^{-1}+t_{2}^{-1}=\frac{t_{1}+t_{2}}{t_{1}t_{2}}=-1\] from (\ref{eqt1t2}) and
\begin{equation}
 (t_{1}^{-1}+1)(t_{2}^{-1}+1)=(t_{1}t_{2})^{-1}+(t_{1}^{-1}+t_{2}^{-1})+1=b^2y_{1}^{-1}.
 \label{t1t2inv}
 \end{equation}

Let $\lambda_{i},\,\mu_{i}$ be the corresponding value of $\lambda,\mu$ when $t=t_{i}$, $i=1,2$. We now show that exactly one of the $\mu_{i}$'s satisfies that $\mu_{i}^{\frac{q+1}{2}}=1$ and the claim follows. We compute that
\begin{align*}
\mu_{i}^{q+1}&=\omega^{\frac{q^2-1}{2}}(\beta^q+t_{i})(\beta+t_{i})(a^2t_{i}^2)^{-1}(t_{i}^{-1}+1)^{-(q-1)}\\
&=-(\beta^{q+1}+(\beta^q+\beta)t_{i}+t_{i}^2)(a^2t_{i}^2)^{-1}\\
&=-(y_{1}+t_{i}y_{1}+t_{i}^2)(a^2t_{i}^2)^{-1}\\
&=-(-(b^2-1)t_{i}^2)(a^2t_{i}^2)^{-1}=1.
\end{align*}
for $i=1,\,2$. Here we have made use of Eqn. (\ref{y1eqbetaqp1}), Eqn. (\ref{y1eqTrbeta}), Eqn. (\ref{1paseqba}) and the fact that $a,\,t_{i}\in\mathbb{F}_{q}$, $b^2t_{i}^2+y_{1}t_{i}+y_{1}=0$ for $i=1,\,2$. Hence, we have $\mu_{i}^{\frac{q+1}{2}}\in\{1,\,-1\}$ for $i\in\{1,\,2\}$.

We next compute that
\begin{align*}
\mu_{1}^{\frac{q+1}{2}}\mu_{2}^{\frac{q+1}{2}}&=\left(\frac{(\beta+t_{1})\omega^{\frac{q-1}{2}}}{at_{1}(t_{1}^{-1}+1)^{\frac{q-1}{2}}}\cdot\frac{(\beta+t_{2})\omega^{\frac{q-1}{2}}}{at_{2}(t_{2}^{-1}+1)^{\frac{q-1}{2}}}\right)^{\frac{q+1}{2}}\\
&=\omega^{\frac{q^2-1}{2}}((t_{1}^{-1}+1)(t_{2}^{-1}+1))^{-\frac{q-1}{2}}\left(\frac{\beta^2+(t_{1}+t_{2})\beta+t_{1}t_{2}}{a^2t_{1}t_{2}}\right)^{\frac{q+1}{2}}\\
&=-y_{1}^{\frac{q-1}{2}}\left[(\beta^2-y_{1}b^{-2}(\beta-1))\left({a^2y_{1}b^{-2}}\right)^{-1}\right]^{\frac{q+1}{2}}\\
&=-y_{1}^{\frac{q-1}{2}}\left[(\beta^2-y_{1}b^{-2}\beta^{-(q-1)})(a^2y_{1}b^{-2})^{-1}\right]^{\frac{q+1}{2}}\\
&=-y_{1}^{\frac{q-1}{2}}\left((b^2-1)\beta^2(a^{2}y_{1})^{-1}\right)^{\frac{q+1}{2}}\\
&=-y_{1}^{\frac{q-1}{2}}\cdot y_{1}\cdot y_{1}^{-\frac{q+1}{2}}=-1.
\end{align*}
Here, we have made use of Eqn. (\ref{t1t2inv}), Eqn. (\ref{eqt1t2}), Eqn. (\ref{betam1eqbetaqm1}) and Eqn. (\ref{1paseqba}).

Therefore, there is exactly one $\mu_{i}$ such that $\mu_{i}^{\frac{q+1}{2}}=1$ for $i\in\{1,2\}$.
\end{enumerate}

To sum up, we deduce that $|\ell\cap\mathcal{M}|=\frac{q-1}{2}$, and the proof for Case 3 is complete.\\

\textbf{Case 4.} The line $\ell$ of $Q(4,q)$ passes through $P=(1,0,\omega^{\frac{3(q-1)}{2}},1)$. This case is almost the same as Case 3 and we omit the details.\\

To sum up, each line of $Q(4,q)$ intersects $\mathcal{M}$ in $\frac{q-1}{2}$ points. Thus, $\mathcal{M}$ is a $\frac{q-1}{2}$-ovoid of $Q(4,q)$. This completes the proof.
\end{proof}

\begin{remark}
We define an isometry of order 2 of $Q(4,q)$ as follows: $$\sigma: (x,y,\alpha,z)\rightarrow (y,x,\alpha,z)$$
Then $\sigma$ stabilizes our $\frac{q-1}{2}$-ovoid $\mathcal{M}$ by direct check, and $\langle\sigma\rangle$ normalizes $G$. For $q=9,\,13,\,17$, we have checked with Magma \cite{cannon2006handbook} that $\langle G,\sigma\rangle$, which is isomorphic to $C_{\frac{q^2-1}{2}}\rtimes(C_{2}\times C_{2})$, is the full stabilizer of $\mathcal{M}$ in $\text{PGO}(5,q)$.
\end{remark}
\section{Concluding remarks}
In this paper, we have constructed $\frac{q-1}{2}$-ovoids in $Q(4,q)$ for $q\equiv 1 (\text{mod}\ 4)$, $q>5$. Together with the results in \cite{feng2016family} and \cite{bamberg2009tight}, this shows that $\frac{q-1}{2}$-ovoids exist in $Q(4,q)$ for all odd prime power $q$. Our technique is similar to that in \cite{feng2016family} and our main contribution is that we find the correct prescribed automorphism group to make the technique in \cite{feng2016family} applicable in our case. This was a challenge due to the rich subgroup structure of $\text{PGO}(5,q)$. The determination of the spectrum of $m$-ovoids in $Q(4,q)$, i.e., to determine for which $m$ there is an $m$-ovoid in $Q(4,q)$, seems out of reach for the moment.

\section*{Acknowledgement}
This work was supported by National Natural Science Foundation of China under
Grant No. 11771392.


\end{document}